\documentclass[a4paper,11pt,oneside,reqno]{amsart}%

\usepackage[T1]{fontenc}
\usepackage[latin1]{inputenc}
\usepackage[british]{babel}
\usepackage{amsmath, amsfonts, amssymb, amsthm, amscd, mathtools}
\usepackage{geometry}
\usepackage{times}
\usepackage{mathrsfs}
\usepackage{dsfont}
\usepackage[inline]{enumitem}
\usepackage{wasysym}
\usepackage[colorlinks=true, linkcolor=black]{hyperref}
\usepackage[square,numbers]{natbib}

\usepackage[monochrome]{color}
\definecolor{red}{rgb}{1.0,0.0,0.0}

\definecolor{blu}{rgb}{0.0,0.0,1.0}
\def\blu#1{{\textcolor{blu}{#1}}}
\definecolor{gre}{rgb}{0.03,0.50,0.03}

\definecolor{darkviolet}{rgb}{0.58, 0.0, 0.83}

\definecolor{airforceblue}{rgb}{0.36, 0.54, 0.66}
\definecolor{greencs}{rgb}{0.0, 0.62, 0.42}
\definecolor{bleu}{rgb}{0.19, 0.55, 0.91}
\definecolor{mossgreen}{rgb}{0.68, 0.87, 0.68}
\definecolor{lb}{rgb}{0.45, 0.76, 0.98}
\definecolor{lr}{rgb}{1.0, 0.41, 0.38}

\providecommand{\U}[1]{\protect\rule{.1in}{.1in}}

\allowdisplaybreaks

\makeatletter
\newtheorem*{rep@theorem}{\rep@title}
\newcommand{\newreptheorem}[2]{%
\newenvironment{rep#1}[1]{%
 \def\rep@title{#2 \ref{##1}}%
 \begin{rep@theorem}}%
 {\end{rep@theorem}}}
\makeatother

\makeatletter
\newcommand{\myitem}[1]{
\item[#1]\protected@edef\@currentlabel{#1}
}
\makeatother

\newtheorem{theorem}{Theorem}[section]
\newtheorem*{theorem*}{Theorem}

\newtheorem{corollary}[theorem]{Corollary}

\newtheorem{definition}[theorem]{Definition}
\newtheorem*{definition*}{Definition}

\newtheorem{lemma}[theorem]{Lemma}

\newtheorem{proposition}[theorem]{Proposition}

\theoremstyle{definition}

\newtheorem{remark}[theorem]{Remark}
\newtheorem{assumption}[theorem]{Assumption}

\newcommand{\bE}{\mathds{E}}

\newcommand{\bR}{\mathds{R}}

\newcommand{\rL}{\mathscr{L}}

\newcommand{\rO}{\mathscr{O}}

\newcommand{\rT}{\mathscr{T}}

\newcommand{\sL}{\mathcal{L}}

\newcommand{\bfb}{{\ensuremath{\mathbf b}} }

\newcommand{\bfx}{{\ensuremath{\mathbf x}} }
\newcommand{\bfy}{{\ensuremath{\mathbf y}} }

\newcommand{\bfX}{{\ensuremath{\mathbf X}} }
\newcommand{\bfY}{{\ensuremath{\mathbf Y}} }

\newcommand{\dd}{{\ensuremath{\mathrm d}} }
\newcommand{\de}{{\ensuremath{\mathrm e}} }

\newcommand{\dC}{{\ensuremath{\mathrm C}} }
\newcommand{\dD}{{\ensuremath{\mathrm D}} }

\newcommand{\ud}{\,\mathrm{d}}

\newcounter{sqindex}

\renewcommand{\epsilon}{\varepsilon}
\renewcommand{\phi}{\varphi}

\DeclarePairedDelimiter{\abs}{\lvert}{\rvert}
\DeclarePairedDelimiter{\norm}{\lVert}{\rVert}
\DeclarePairedDelimiterX{\inprod}[2]{\langle}{\rangle}{#1, #2}

\newcommand{\R}{\bR}

\hypersetup{pdftitle={Optimal control SIRD},pdfauthor={Alessandro Calvia, Fausto Gozzi, Francesco Lippi, Giovanni Zanco},%
pdftex,pdfstartview={XYZ null null 1.4},colorlinks=true,linkcolor=black,urlcolor=blue}

\newcommand{\mail}[1]{\href{mailto:#1}{\normalfont\texttt{#1}}}

\makeatletter
\def\@setthanks{\vspace{-\baselineskip}\def\thanks##1{\@par##1\@addpunct.}\thankses}
\makeatother

\title[COVID-19 lockdown: a DP approach]{A simple planning problem for COVID-19 lockdown: a dynamic programming approach}
\author[A.~Calvia ]{Alessandro Calvia\textsuperscript{\MakeLowercase{a},1}}
\thanks{\noindent \textsuperscript{a} Dipartimento di Economia e Finanza, LUISS University, Viale Romania 32, 00197, Rome, Italy.}

\author[F.~Gozzi ]{Fausto Gozzi\textsuperscript{\MakeLowercase{a},2}}

\author[F.~Lippi ]{Francesco Lippi\textsuperscript{\MakeLowercase{a, b},3}}
\thanks{\noindent \textsuperscript{b} Einaudi Institute for Economics and Finance, Via Sallustiana 62, 00187, Rome, Italy.}

\author[G.~Zanco ]{Giovanni Zanco\textsuperscript{\MakeLowercase{a},4}}
\thanks{\noindent \textsuperscript{1} E-mail: \mail{acalvia@luiss.it}.
\\
\noindent \textsuperscript{2} E-mail: \mail{fgozzi@luiss.it}
\\
\noindent \textsuperscript{3} E-mail: \mail{flippi@luiss.it}
\\
\noindent \textsuperscript{4} E-mail: \mail{gzanco@luiss.it}}

\begin{document}

\begin{abstract}
{\color{blue}
A large number of recent studies consider a compartmental SIR model to study  optimal control policies aimed at   containing the diffusion of COVID-19  while minimizing the economic costs of preventive measures.    
 Such  problems are   non-convex and  standard results   need not to hold. 
We use a Dynamic Programming approach and  prove some continuity properties of the value function of the associated optimization problem. 
We study the corresponding Hamilton-Jacobi-Bellman equation and  show that the value function solves it in the viscosity sense.
 Finally, we discuss  some optimality conditions. 
 Our paper represents a first contribution towards a complete analysis of non-convex dynamic optimization problems, within a Dynamic Programming approach.
}

\end{abstract}

\maketitle


\bigskip


\bigskip

\noindent AMS 2020: 49K15, 49L20, 49L25.

\bigskip

\tableofcontents

\allowdisplaybreaks

\section{Introduction}
Following the COVID-19 outbreak a large number of  papers have been written combining elements of both epidemiology and economics. One important motivation for these papers is that the pandemic faced policymakers with the challenge of keeping the virus diffusion under control while avoiding to suffocate economic activity (see, e.g., \citet{acemoglu2021,AAL21,Atk20,ERT21,federicoferrari2021}).
{\color{blue}From a mathematical perspective, this motivation leads to the formulation of suitable dynamic optimization problems, that can be tackled with different techniques. A common issue of these problems is that,} even in simple {\color{blue}settings}, without considering heterogeneity of viral transmission or uncertainty on the infection mortality, they are {\color{blue}mathematically} involved because they may not be convex.

\textcolor{blue}{In the typical setup of these problems, the state dynamics  are given by the so-called compartmental models, where  the state variables are the epidemic compartments, such as the Susceptibles, the Infected, and the Recovered in the SIR  model. 
A peculiar feature of the state equations is an interaction term between some of these classes, usually the product between the number of susceptibles and infected. The state dynamics provide the constraint for an optimization problem and the interaction term makes the Hamiltonian function associated to the control problem non-convex. In this situation, the classical sufficiency results for the Pontryagin Maximum Principle, like the Arrow or Mangasarian conditions, cannot be applied, as they require convexity of the Hamiltonian function.
Likewise, the lack of convexity hinders the application of numerical methods, which are often employed in the analysis of these dynamic optimization problems.
}

\blu{This paper studies a specific family of non-convex problems, to show that the Dynamic Programming approach can be profitably applied to analyze them.
In particular, Dynamic Programming allows us to characterize the value function of the optimization problem as the unique viscosity solution
of a suitable Hamilton-Jacobi-Bellman (HJB) equation. It also identifies 
optimality conditions which are sufficient for global optima independently of any convexity assumption (see, e.g.,~\citet{bardi1997:optctrl}).
}

  \textcolor{blue}{
We illustrate our approach with an application to  the  simple model of~\citet{AAL21}, one of the recent  epi-econ  papers  featuring the type of non-convexities in the epidemic propagation terms discussed above.
The main results of the paper establish:
the continuity, and in some cases the Lipschitz continuity, of the value function of the optimization problem; the fact that the
value function is the unique viscosity solution of the HJB equation associated to the control problem
and that it is also a bilateral viscosity solution; a sufficient optimality condition in terms of the semidifferentials
of the value function. More details are given in Subsections~\ref{sec:keycontrib} and~\ref{sec:mainres} below.
It is important to highlight that our results can be proved (with straightforward changes) for other models with similar or more general features, such as those presented in~\citet{acemoglu2021,PigShi22,ERT21,FJS21}.
}


\blu{We emphasize that, due to the difficulty of the problem, some issues remain  open.
In particular, we cannot prove that the value function is differentiable everywhere.
This also means that we cannot rule out a singular behaviour (e.g., discontinuities) of optimal strategies in some regions of the state space.
This fact may represent an issue, for instance, in numerical simulations for the type of optimization problems that we consider.
To the best of our knowledge, there are no results in the mathematical literature that provide a numerical scheme suited to them; in addition, the extension of known numerical schemes for viscosity solutions to the kind of first-order Hamilton-Jacobi-Bellman equations that we treat here does not seem straightforward, due again to the non-convexity of the state dynamics and of the Hamiltonian.}

{\color{blu}
\subsection{Technical issues and selected related contributions}\label{sec:keycontrib}
Optimal control problems that are non-convex either in the objective functional
or in the state dynamics (or both), are notoriously difficult to study with a
Maximum Principle approach. Indeed, the standard sufficiency conditions,
like the Arrow or the Mangasarian conditions, do not hold.
Nonetheless,~\citet{goenkaetal2021sir, goenkaetal2022modelling} analyzed epi-econ SIS and SIR models with a Maximum Principle approach, by providing sufficient conditions for local extrema under weaker assumptions than those of the Arrow or the Mangasarian conditions. The Maximum Principle approach was applied also in~\citet{goenkaetal2014} to a SIS model complemented with economic variables, by proving existence of optimal strategies.
The results proved in all these papers rely on the structure of the problems, in particular on the convexity
of the objective function and, for sufficient conditions, on some ideas given first in~\citet{leitmannstalford1971}.
Unfortunately, these results cannot be directly applied to our case, since our objective function may be non-convex.

As we previously anticipated, the Dynamic Programming approach (if applicable) presents clear advantages in treating non-convex problems.
One of these advantages is the possibility of characterising the value function of the optimization problem as the unique viscosity solution\footnote{The concept of viscosity solutions has been introduced by Crandall and Lions in~\citep{CrandallLions1983:viscsol} (see, e.g.,~\citet{crandalletal1992:users} for a synthesis of viscosity solution theory) to cope with the fact that, in many optimal control problems, the value functions are not differentiable everywhere and the associated  HJB equations may not have classical (i.e., differentiable) solutions, even in simple cases.
Using such a more general solution concept it is possible to prove existence and uniqueness of solutions which are simply continuous (or even discontinuous) and to apply suitable algorithms to compute numerically the value function.}
of the HJB equation associated to the optimal control problem. It is important to stress that proving such a characterization can motivate the study of numerical algorithms suited to the class of optimization problems that we study here. These numerical schemes could be used to approximate the value function\footnote{We refer the reader to~\citet[Appendix A]{bardi1997:optctrl} for an introduction to numerical schemes for viscosity solutions to HJB equations.}.
In our setting, the main issue that we face to prove this property of the value function is the presence of positivity state constraints, which may be a hindrance to show uniqueness, see, e.g.,~\citet{soravia1999a, soravia1999b}. This problem is solved because the so-called interior cone condition holds. This condition was introduced first in~\citet{soner:optcontrol1} and it allows us to prove uniqueness of the solution to the HJB equation in the viscosity sense (see Theorem~\ref{th:Vconstrviscsol} below).

Another advantage of the Dynamic Programming approach is the possibility of identifying optimality conditions which are sufficient for global optima
independently of any convexity assumptions (more details on this in, e.g.,~\citet[Chapter III, Section 2.5]{bardi1997:optctrl} and~\citet[Section 2.5]{fabbri:soc}). These conditions are usually obtained through the so-called Verification Theorems and the main issue is to show that the value function is continuously differentiable in the interior of the state space.
This is rather problematic in our setting because, in general, value functions are not continuously differentiable, i.e., they admit singularities
(see, e.g.,~\citet[Section II.2]{flemingsoner2006}).
There are quite general conditions that imply continuous differentiability of the value function, namely, its semiconcavity and strict convexity in the costate variables of the Hamiltonian function (see~\citet{cannarsasoner1989}, \citet[Section 5.3]{cannarsasinestrari2004} and also~\citet[Chapter II, Section 5]{bardi1997:optctrl}).
Unfortunately in our case these conditions do not hold or are difficult to show (provided that they hold).
For this reason, based on the ideas of~\citet[Chapter III, Sections 2.3-2.4]{bardi1997:optctrl}, we establish a Backward Dynamic Programming Principle. This is key to prove that the value function is a bilateral solution of the correspoding HJB equation (see Theorem~\ref{th:V_viscsol} below) and to state a weak form of sufficient optimality condition (see Theorem~\ref{thm:optimal}). 

We mention that similar techniques were used successfully in other economic applications,
see, e.g.,~\citet{bambigozzi2020, frenigozzisalvadori2006, frenigozzipignotti2008:optstrat}. However, in those problems homogeneity and semiconcavity allowed the authors to apply the method of~\citet{cannarsasoner1989}, which cannot be used in our setting.
}

\subsection{Overview of the main results}\label{sec:mainres}
{\color{blue}
From a technical perspective, we can single out three main contributions of our paper.
\begin{itemize}
  \item 
First, we prove that the value function is uniformly continuous and, for a sufficiently large discount rate, Lipschitz continuous in its domain (see Proposition~\ref{prop:V_UCb}).
  \item 
Second, we establish dynamic programming principles
for our problem (the standard one and the backward one, see Propositions~\ref{prop:dpp} and~\ref{prop:BDP}, respectively). This allows us to characterize the value function as the unique viscosity solution to the associated HJB equation~\eqref{eq:HJB}, satisfying a suitable boundary condition (i.e., being a supersolution at the boundary),
and to prove that it is also a bilateral solution (Theorem~\ref{th:V_viscsol}).
  \item 
Third, we show an optimality condition (see Theorem~\ref{thm:optimal}), that allows us to characterize the optimal strategies. In particular, we show that (except from trivial cases) the optimal strategy is a \textit{laissez-faire} policy as long as the ratio between the rate of newly infected people and the population that can be put in lockdown is not higher than a threshold, which depends on the difference between the marginal cost of infected and the marginal cost of susceptibles. As this ratio increases, the lockdown is enforced up to a full lockdown, when a second threshold is passed.
\end{itemize}
}

\bigskip
The paper is organized as follows.
In Section~\ref{sec:model} we introduce the optimal control problem for the SIRD model that we aim to analyse {\color{blue} and we provide some preliminary results.}
{\color{blue}In Section~\ref{sec:dp} we provide continuity properties of the value function (Subsection~\ref{sec:Vprop}) and we establish the dynamic programming principles (Subsection~\ref{sec:dpprinc}).
In Section~\ref{sec:HJB}} we study the HJB equation and, in particular, we provide the explicit expression of the Hamiltonian function; in Subsection~\ref{sec:Vvisc} we prove that the value function is a viscosity solution, in a suitable sense, of the HJB equation; Subsection~\ref{sec:opt} contains some optimality conditions{\color{blue}, with which we are able to provide an economic intepretation of optimal policies.}
In Section~\ref{sec:concl} we draw some conclusions on our results and present some ideas for future work on the subject.

\section{The optimal control problem}
\label{sec:model}
In this section we introduce the optimization problem for the SIRD model that we study.
{\color{blue}We denote by $S$, $I$, $R$, $D$, the classes of susceptible, infectious, recovered, and dead individuals, respectively}. We assume that there are no newborn and that people either die from the infection or live forever; this is clearly unrealistic{\color{blue}, but it} is compatible with the duration of the pandemic/endemic phase, {\color{blue}which is} shorter than the average life duration.

{\color{blue}The dynamics of the population introduced above are influenced by a planner, who may enforce lockdown by choosing its intensity, i.e., the percentage $L_t$ of the population that is forced to be locked down, at each time $t \geq 0$. This percentage can be chosen up to some fixed threshold $\bar L \leq 1$, that is, $L_t \in [0,\bar L]$, for each $t \geq 0$.
However,} the lockdown effectiveness is assumed to be less than {\color{blue}the planned one}, because people may fail to respect the lockdown measures and the virus can still circulate; the lockdown intensity is thus damped by a factor $\theta\in(0,1)${\color{blue}, i.e., $\theta L_t$ is the real fraction of population that is actually in lockdown.
Lockdown applies only to susceptible and infectious individuals, since} we assume that recovered ones cannot get infected again; this is possible because we assume that testing is available, and hence the planner knows who is infected and who has recovered.
{\color{blue}Finally,} we account for the possibility of a vaccine and a cure being discovered (for simplicity at the same time) {\color{blue}and we assume} that the epidemic dynamics are instantaneously stopped; such vaccine and cure appear at a random time $\tau$, which we assume to be exponentially distributed with intensity $\nu$.

The model we consider is specified as follows. For $0 \leq t < \tau$, the process $(S,I,R,D)$ evolves according to the following system of controlled ordinary differential equations
\begin{equation}
\label{eq:SIRD}
\left\{
\begin{aligned}
S^\prime_t&=-\beta(1-\theta L_t)S_t (1-\theta L_t )I_t, & S_0 &= s_0, \\
I^\prime_t &=\beta(1-\theta L_t )S_t (1-\theta L_t )I_t -\gamma I_t -I_t \phi\left(I_t \right), & I_0 &= i_0,\\
R^\prime_t &=\gamma I_t, & R_0 &= r_0, \\
D^\prime_t &=I_t \phi\left(I_t \right), & D_0 &= d_0.\\
\end{aligned}
\right.
\end{equation}
{\color{blue}The lockdown intensity function} $t \mapsto L_t$ is chosen in the set of admissible control strategies
\begin{equation}\label{eq:admissible_ctrl}
  \rL \coloneqq \left\{L\colon[0,+\infty)\to [0,\bar L], \text{ Borel-measurable }\right\}.
\end{equation}
{\color{blue}The parameters appearing in~\eqref{eq:SIRD} have the following meaning:}
$\beta > 0$ is the number of susceptible agents per unit of time to whom an infected agent can transmit the virus, among those who are not in lockdown; $\gamma > 0$ is the fraction of infected agents that recovers; $\phi(i)$ is the rate per unit of time of infected agents $i$ that die.

At time $\tau$, processes $S$ and $I$ jump to $0$, process $R$ jumps to $R_{\tau^-} + S_{\tau^-} + I_{\tau^-}$, while process $D$ remains at the same level immediately prior to $\tau$, i.e., $D_{\tau} = D_{\tau^-}$. More precisely, we have that, for $t \geq \tau$,
\begin{align}\label{eq:SIRDaftertau}
&S_t = I_t = 0, & &R_t = R_{\tau^-} + S_{\tau^-} + I_{\tau^-}, & &D_t = D_{\tau^-}.
\end{align}
{\color{blue}We postulate that there are no more susceptible or infected people after a vaccine and a cure arrive at time $\tau$}; all susceptible and infected immediately prior to $\tau$ {\color{blue}recover istantaneously and} there are no new deaths.

{\color{blue}It is worth emphasizing that the reason why we model the arrival of a vaccine and a cure as above is justified by the tractability of the optimization problem that we introduce below. Indeed, in this problem the planner controls the epidemic dynamics up to the (random) time $\tau$ at which a vaccine and a cure are discovered. In other words, the planner disregards what happens to the epidemic dynamics after time $\tau$.}

\begin{remark}
The case {\color{blue}where the highest possible intensity of lockdown is equal to $1$, i.e., $\bar L = 1$, corresponds to allowing the possibility of a full lockdown}. This is not realistic, as basic activities related for example to energy production and distribution of fundamental goods must remain functional. {\color{blue}Nonetheless, we will not introduce the restriction $\bar L<1$, since it} has no particular effect on the mathematical results presented below.
\end{remark}

{\color{blue}
The following assumption ensures existence and uniqueness of a solution to (\ref{eq:SIRD}), for any given {$L \in \rL$}. This can be easily shown with standard methods {(see, e.g., \citep[Chapter~III, Section~5]{bardi1997:optctrl})}.

\begin{assumption}\label{as:phi}
The function $\phi$, appearing in~\eqref{eq:SIRD}, is positive, bounded, and Lipschitz continuous. More specifically,
\begin{equation}\label{eq:phibounded}
0 < \phi(i) \leq \gamma, \quad \text{for all } i \in [0,1],
\end{equation}
where $\gamma$ is the fraction of infected agents that recovers,
and there exists a constant $M_\phi > 0$, such that, for all $i, i' \in [0,1]$,
\begin{equation}\label{eq:philip}
\abs{\phi(i) - \phi(i')} \leq M_\phi \abs{i - i'}.
\end{equation}
\end{assumption}

\begin{remark}\label{rem:phi}
In our model, the mortality rate is not constant, but provided by the function $\phi$, and depends on the number of infected people. Such a choice is motivated by some studies (see, e.g., \citep{ciminelli2020healthcare,favero2020covid}). Various papers in the literature deal, instead, with a constant mortality rate (lower than $\gamma$), which is a case covered by our model. However, specializing our results to the case of a constant mortality rate would not allow us to obtain deeper or more refined statements. Indeed, as highlighted in the Introduction, the technical difficulties lie in the fact that the epidemic dynamics given in~\eqref{eq:SIRD} feature an interaction term between the number of susceptibles and infected.
\end{remark}
}

The planner maximizes the following quantity over all admissible control strategies
\begin{equation}
  \label{eq:J_nu_prelim}
  \bE\left[\int_0^{+\infty} \de^{-rt}\left[\left(N_t -D_t -\left(S_t +I_t \right)L_t \right)w-\chi I_t \phi\left(I_t \right)\right]\ud t\right] \, ,
\end{equation}
where $r > 0$ is a fixed discount factor, $w > 0$ is the output produced by each agent alive that is not in lockdown, $\chi > 0$ is an extra cost, in units of output, for each agent that dies as a consequence of the infection.
{\color{blue}Hence, the planner aims at maximizing the present value of the total production output, considering the cost of fatalities.}

\subsection{{\color{blue}Preliminary results}}\label{sec:prelim}
{\color{blue}Let us fix, for the time being, an arbitrary admissible control $L \in \rL$. Let us denote by $N_t$ the total population (including deaths) at time $t \geq 0$, i.e., $N_t \coloneqq S_t +I_t +R_t +D_t$.}
For simplicity, we normalize the initial population so that $N_0=s_0+i_0+r_0+d_0=1$.

{\color{blue}From~\eqref{eq:SIRD}}, we have $N^\prime_t =0$, thus $N_t = 1$, for all $t \geq 0$.
Therefore, for every time $t \geq 0$ and any initial condition $(s_0,i_0,r_0,d_0)$, with $s_0+i_0+r_0+d_0=1$,
the state of the system $(S_t,I_t,R_t,D_t)$ belongs to the simplex\footnote{Said otherwise, the flow associated to~\eqref{eq:SIRD} leaves invariant the set $\Delta$.}
\begin{equation*}
  \Delta \coloneqq \left\{(s,i,r,d)\colon s,i,r,d\geq 0,\ s+i+r+d=1\right\}.
\end{equation*}
{\color{blue}This fact is consistent with the assumption that there are no newborn and that people either die from the infection or live forever.
Moreover, to determine uniquely the solution to~\eqref{eq:SIRD} it is enough to provide the triplet $(s_0, i_0, r_0)$ as initial condition, with $s_0, i_0, r_0 \geq 0$ and $s_0 + i_0 + r_0 \leq 1$, and set $d_0 = 1 - s_0 - i_0 - r_0$. From now on, we will specify only such a triplet, unless stated otherwise.}

Since $t\mapsto D_t $ is clearly nondecreasing, we have that the number of people alive at time $t \geq 0$, i.e., $N_t - D_t = S_t + I_t + R_t$, is {\color{blue}non}decreasing over time{\color{blue}, that is,
\begin{equation*}
S_t+I_t+R_t\leq s_0+i_0+r_0 \leq 1, \quad t \geq 0.
\end{equation*}
Therefore, for all $t \geq 0$ and any initial condition $(s_0, i_0, r_0)$ as above,%
}
the state of the system $(S_t,I_t,R_t,D_t)$ belongs to the set
\begin{equation*}
\Delta(s_0,i_0,r_0) \coloneqq \{(s,i,r,d) \in \Delta \colon s + i +r \leq s_0 + i_0 + r_0\}.
\end{equation*}

\begin{remark}\label{zeroinit}
It is immediate to check that if $i_0 =0$, then $(S_t,R_t,D_t)=(s_0,r_0,d_0)$, for every $t\geq 0$, i.e., the dynamics is constant and not affected by the choice of the control strategy. If $s_0=0$ the dynamics is not constant but, as before, is not affected by the choice of the control strategy.
\end{remark}

{\color{blue}Recall that the planner maximizes over all admissible control strategies $L \in \rL$ the functional
\begin{equation}
  \label{eq:J_nu}
 \tilde J(L, s_0, i_0, r_0) \coloneqq
  \bE\left[\int_0^{+\infty} \de^{-rt}\left[\left(N_t -D_t -\left(S_t +I_t \right)L_t \right)w-\chi I_t \phi\left(I_t \right)\right]\ud t\right] \, ,
\end{equation}
which depends on any given initial condition $s_0$, $i_0$, $r_0 \geq 0$, with $s_0 + i_0 + r_0 \leq 1$, for~\eqref{eq:SIRD}.%
}
Using the dynamics of processes $(S,I,R,D)$ and the fact that $\tau$ is an exponential random variable with intensity $\nu$, we can rewrite the functional $\tilde J$ as follows.
{\color{blue}
\begin{lemma}\label{lem:Jdeterministic}
For all $s_0$, $i_0$, $r_0 \geq 0$, with $s_0 + i_0 + r_0 \leq 1$, and all $L \in \rL$, it holds
\begin{equation}
  \label{eq:Jold_Jnew_equiv}
  \tilde J(L,s_0,i_0,r_0)=\dfrac wr - \int_0^{+\infty} \de^{-(r+\nu)t}\left(\left(S_t +I_t \right)L_t w+\left(\frac{w}{r}+\chi\right) I_t \phi\left(I_t \right)\right)\ud t \, .
\end{equation}
\end{lemma}

\begin{proof}
We can explicitly compute the expectation in~\eqref{eq:J_nu} using the law of random time $\tau$ and the dynamics of $(S,I,R,D)$.
\begin{align*}
&\mathop{\phantom{=}}\tilde J(L,s_0,i_0,r_0)=\bE\left[\int_0^{+\infty} \de^{-rt}\left(\left(N_t -D_t -\left(S_t +I_t \right)L_t \right)w-\chi I_t \phi\left(I_t \right)\right)\ud t\right]
\\
&=\bE\left[\int_0^{\tau} \de^{-rt}\left(\left(N_t -D_t -\left(S_t +I_t \right)L_t \right)w-\chi I_t \phi\left(I_t \right)\right)\ud t \right.
\\
&\qquad \qquad + \left. \int_{\tau}^{+\infty} \de^{-rt}\left(\left(N_{\tau^-} -D_{\tau^-}\right)w\right)\ud t \right]
\\
&= \int_0^\infty \left\{\int_0^s \de^{-rt}\left(\left(N_t -D_t -\left(S_t +I_t \right)L_t \right)w-\chi I_t \phi\left(I_t \right)\right)\ud t \right\} \nu \de^{-\nu s} \, \dd s
\\
&\qquad \qquad + \int_0^{+\infty} \left\{\int_s^{+\infty} \de^{-rt}\left(\left(N_s -D_s\right)w\right)\ud t\right\} \nu \de^{-\nu s} \, \dd s
\end{align*}
Applying the Fubini-Tonelli's theorem we get that
\begin{align*}
&\mathop{\phantom{=}} \int_0^\infty \left\{\int_0^s \de^{-rt}\left(\left(N_t -D_t -\left(S_t +I_t \right)L_t \right)w-\chi I_t \phi\left(I_t \right)\right)\ud t \right\} \nu \de^{-\nu s} \, \dd s
\\
&\qquad \qquad + \int_0^{+\infty} \left\{\int_s^{+\infty} \de^{-rt}\left(\left(N_s -D_s\right)w\right)\ud t\right\} \nu \de^{-\nu s} \, \dd s
\\
&= \int_0^\infty \left\{\int_t^{+\infty} \nu \de^{-\nu s} \, \dd s\right\}  \de^{-rt}\left(\left(N_t -D_t -\left(S_t +I_t \right)L_t \right)w-\chi I_t \phi\left(I_t \right)\right)\ud t
\\
&\qquad \qquad +\int_0^\infty \left\{\int_s^{+\infty} \de^{-rt} \, \dd t \right\} \nu \de^{-\nu s} (N_s - D_s)w \, \dd s \\
&=\int_0^{+\infty} \de^{-(r+\nu)t}\left(\left(N_t -D_t\right) \dfrac{r+\nu}{r} w -\left(S_t +I_t \right) w L_t -\chi I_t \phi\left(I_t \right)\right)\ud t \, .
\end{align*}
Finally, noting that $N^\prime_t - D^\prime_t = S^\prime_t + I^\prime_t + R^\prime_t = I_t\phi(I_t)$ and integrating by parts, we get
\begin{align*}
&\mathop{\phantom{=}} \tilde J(L,s_0,i_0,r_0)=\int_0^{+\infty} \de^{-(r+\nu)t}\left(\left(N_t -D_t\right) \dfrac{r+\nu}{r} w -\left(S_t +I_t \right) w L_t -\chi I_t \phi\left(I_t \right)\right)\ud t
\\
&=\dfrac{r+\nu}{r} w \int_0^{+\infty} \de^{-(r+\nu)t}\left(N_t -D_t\right) \, \dd t - \int_0^{+\infty}\de^{-(r+\nu)t} \left(\left(S_t +I_t \right) w L_t +\chi I_t \phi\left(I_t \right)\right)\ud t
\\
&= \dfrac wr N_0 - \dfrac wr \int_0^{+\infty} \de^{-(r+\nu)t} I_t \phi(I_t) \, \dd t - \int_0^{+\infty}\de^{-(r+\nu)t} \left(\left(S_t +I_t \right) w L_t +\chi I_t \phi\left(I_t \right)\right)\ud t,
\end{align*}
whence, recalling that $N_0 = 1$, we obtain~\eqref{eq:Jold_Jnew_equiv}.
\end{proof}

\begin{remark}
The state variables appearing on the right-hand-side of~\eqref{eq:Jold_Jnew_equiv} are only $S$ and $I$, i.e., the number of susceptible and infectious individuals. If we consider the dynamics of this pair of variables%
, namely,
\begin{equation}
\label{eq:SIRDreduced}
\left\{
\begin{aligned}
S^\prime_t&=-\beta(1-\theta L_t)S_t (1-\theta L_t )I_t, & S_0 &= s_0\\
I^\prime_t &=\beta(1-\theta L_t )S_t (1-\theta L_t )I_t -\gamma I_t -I_t \phi\left(I_t \right), & I_0 &= i_0,
\end{aligned}
\right.
\end{equation}
then, for each $L \in \rL$, the solution to~\eqref{eq:SIRDreduced} is completely determined by any given initial condition $(s_0, i_0) \in \rT$, where
\begin{equation}\label{eq:Tdef}
\rT \coloneqq \{(s,i) \in [0,1]^2, \text{ s.t. } 0 \leq s+i \leq 1\}.
\end{equation}
Moreover, since the map $t \mapsto S_t + I_t$ is decreasing, we get that $S_t + I_t \leq s_0 + i_0$. Therefore, for any $t \geq 0$, any $(s_0, i_0) \in \rT$, and any $L \in \rL$, the state $(S_t, I_t)$ belongs to the set
\begin{equation*}
\rT(s_0, i_0) \coloneqq \{(s,i) \in \rT \colon s+i \leq s_0+i_0\}.
\end{equation*}
Clearly, specifying only $(s_0, i_0) \in \rT$ is not enough to determine the solution to the complete system~\eqref{eq:SIRD}, as $r_0$ is also needed. As a consequence of the discussion above and of Lemma~\ref{lem:Jdeterministic}, the functional $\tilde J$ does not depend on $r_0$, and hence our optimization problem depends only on the state variables $S$ and $I$.
\end{remark}

Equation~\eqref{eq:Jold_Jnew_equiv} also shows that the optimization problem introduced at the beginning of this section is equivalent to the optimization problem defined, for any given $(s,i) \in \rT$, as
\begin{equation}
\label{eq:optpb}
\tag{$P$}
\begin{aligned}
&\inf_{L\in\rL}J(L, s,i) \\
&\text{s.t. }
\left\{
  \begin{aligned}
    S^\prime_t&=-\beta(1-\theta L_t)S_t (1-\theta L_t )I_t, & S_0 = s,\\
    I^\prime_t&=\beta(1-\theta L_t )S_t (1-\theta L_t )I_t -\gamma I_t -I_t \phi\left(I_t \right), & I_0 = i,
  \end{aligned}
\right.
\end{aligned}
\end{equation}
where, for all $L \in \rL$ and $(s,i) \in \rT$,
\begin{equation}
  \label{eq:J_new}
  J(L,s,i) \coloneqq \int_0^{+\infty} \de^{-(r+\nu)t}\left(\left(S_t +I_t \right)L_t w+\left(\frac{w}{r}+\chi\right) I_t \phi\left(I_t \right)\right)\ud t \, .
\end{equation}
}%
Thus, from this point onward, we consider {\color{blue}problem~\eqref{eq:optpb}}. As usual, we introduce the value function for the above minimization problem, namely,
\begin{equation}\label{eq:valuefunction}
	V(s,i) \coloneqq \inf_{L\in\rL}J(L, s,i), \quad (s,i) \in \rT.
\end{equation}

{\color{blue}
To conclude this section, we provide a brief comparison between the optimization problem studied here, i.e., problem~\eqref{eq:optpb}, and the one analyzed in~\citep{AAL21}.
The optimization problem~\eqref{eq:optpb} is equivalent to the one studied in~\citep[Eq.~(7)]{AAL21}. We modified slightly the setup of~\citep{AAL21} including a class $D$ of dead people, that is, we consider an SIRD model; the dynamics of the classes $S$, $I$ and $R$ are left unchanged, thus our formulation is completely equivalent to that of~\citep{AAL21}, with the advantage that the population remains constant in our setting. We stress once more that we consider here the situation where testing is available and quarantine is not enforced. We make some remarks on other possible extensions of this model further below (see Remark~\ref{rem:estensioni}).

\begin{remark}
It is also worth noting that Lemma~\ref{lem:Jdeterministic} shows that optimization problem~\eqref{eq:optpb} is equivalent to the maximization one presented in~\citep[Eq.~(5)]{alvarez2020simple}. This provides a rigorous foundation to the arguments given in~\citep[p.~10]{alvarez2020simple} regarding the equivalence of these two problems.
\end{remark}
}

\section{{\color{blue}Properties of the value function and Dynamic Programming Principles}}\label{sec:dp}
{\color{blue}
In this section we begin our analysis of problem~\eqref{eq:optpb} with the dynamic programming approach. We derive in Section~\ref{sec:Vprop} a regularity result for the value function of this problem, defined in~\eqref{eq:valuefunction}; then, in Section~\ref{sec:dpprinc} we establish the forward and backward dynamic programming principles, that are used in Section~\ref{sec:HJB}.
}

\subsection{Properties of the value function}\label{sec:Vprop}
%
As shown in the previous section, the state variables for problem~\eqref{eq:optpb} are the number of susceptible and infected individuals, whose dynamics are given in~\eqref{eq:SIRDreduced}.
{\color{blue} To show some regularity results for the value function $V$ (see Proposition~\ref{prop:V_UCb} below), we need to} provide, first, a useful estimate concerning
{\color{blue} the unique solution to this system of ordinary differential equations.}

In what follows, we set for convenience $\mathbf{X}_t=(S_t,I_t)$, $t \geq 0$, and we introduce the notation $\mathbf{X}_t^{L,\mathbf{x}_0}$, $S_t^{L,s_0,i_0}$, $I_t^{L,s_0,i_0}$, to stress the dependence of the solution to~\eqref{eq:SIRDreduced} on the control strategy $L \in \rL$ and on the initial condition $\mathbf{x}_0=(s_0,i_0) \in \rT$. 
We also define the vector field
\begin{equation}\label{eq:bdef}
\bfb(s,i,l) \coloneqq
\begin{bmatrix}
-\beta si(1-\theta l)^2 \\
\beta si(1-\theta l)^2 -\gamma i - i \phi\left(i \right)
\end{bmatrix}
\, , \quad (s,i) \in \rT, \, l \in [0, \bar L] \, .
\end{equation}
In this way, we can write the system~\eqref{eq:SIRDreduced} as
\begin{equation*}
\begin{dcases}
  \mathbf{X}_t^\prime={\bfb}\left(\mathbf{X}_t,L_t\right), \\
  \bfX_0 = \bfx_0 \in \rT,
\end{dcases}
\end{equation*}
or equivalently, in integrated form,
\begin{equation}\label{eq:SIintegral}
\mathbf{X}_t^{L,\bfx_0} = \mathbf{x}_0 + \int_0^t \bfb\left(\mathbf{X}_s^{L,\bfx_0},L_s\right) \, \dd s, \quad t \geq 0.
\end{equation}


{\color{blue}
We have the following lemma.
\begin{lemma}\label{lem:init_data_dep}
Let $\bfX$ and $\widetilde \bfX$ be the two solutions to~\eqref{eq:SIRDreduced} corresponding to initial conditions $\bfx_0, \widetilde \bfx_0 \in \rT$ and strategies $L,\widetilde{L}\in\rL$, respectively. Then,
  \begin{equation}\label{eq:init_data_SI_gen}
    \left\Vert \bfX_t- \widetilde \bfX_t\right\Vert \leq \left(\left\Vert \bfx_0- \widetilde \bfx_0\right\Vert+ 4\theta(\bar L+1)\int_0^t\left\vert L_r-\widetilde L_r\right\vert\ud r\right) \de^{M_b t}.
  \end{equation}
In particular, if $L = \widetilde{L}$, then,
\begin{equation}\label{eq:init_data_SI}
\norm{\bfX_t-\widetilde \bfX_t} \leq \norm{\bfx_0 - \widetilde \bfx_0} \, \de^{M_b t}, \quad t \geq 0.
\end{equation}
\end{lemma}
}

\begin{proof}
It is easy to show that the vector field $\bfb${\color{blue}, introduced in~\eqref{eq:bdef},} is bounded on $\rT\times[0,\bar L]$ and that it is Lipschitz continuous in $(s,i) \in \rT$, uniformly with respect to $l \in [0, \bar L]$. More precisely,
we have that
\begin{equation}\label{eq:bbounded}
\sup_{(s,i,l) \in \rT \times [0, \bar L]} \norm{\bfb(s,i,l)} \leq 3\left[\beta + \gamma \right] \eqqcolon K_b,
\end{equation}
and, for all $(s,i), \, (s',i') \in \rT$,
\begin{equation}\label{eq:blip}
\sup_{l \in [0, \bar L]} \norm{\bfb(s,i,l) - \bfb(s',i',l)} \leq M_b \norm{(s,i) - (s',i')},
\end{equation}
with $M_b \coloneqq 2\left[\beta + \gamma + M_\phi\right]${\color{blue}, where $M_\phi$ is the Lipschitz constant appearing in~\eqref{eq:philip}.}

Thanks to~\eqref{eq:blip}, we deduce that
{\color{blue}
\begin{equation*}
\norm{\bfX_t - \widetilde \bfX_t} \leq \norm{\bfx_0 - \widetilde \bfx_0} + 4\theta(\bar L+1)\int_0^t\left\vert L_r-\widetilde L_r\right\vert\ud r + M_b \int_0^t \norm{\bfX_s - \widetilde \bfX_s} \, \dd s, \quad t \geq 0,
\end{equation*}
}
and hence a simple application of Gronwall's lemma (see, e.g., \citep[Chapter~III, Section~5]{bardi1997:optctrl}, or \citep[Appendix~A]{flemingrishel1975:soc}) {\color{blue}yields~\eqref{eq:init_data_SI_gen}.
Setting $L = \widetilde L$ we immediately deduce~\eqref{eq:init_data_SI}.}
\end{proof}

{\color{blue}Let us introduce} the running cost function appearing inside the functional $J$ given in~\eqref{eq:J_new}, i.e.,
\begin{equation}\label{eq:runningcost}
f(s,i,l) \coloneqq \left(s + i \right)l w+\left(\frac{w}{r}+\chi\right) i \phi\left(i\right), \quad (s,i,l) \in \rT \times [0, \bar L].
\end{equation}
It is not hard to show that $f$ is non-negative and bounded on $\rT \times [0, \bar L]$ and that it is Lipschitz continuous in $(s,i)\in \rT$, uniformly with respect to $l\in[0,\bar L]$. {\color{blue}More precisely,}
\begin{equation}\label{eq:fbounded}
\sup_{(s,i,l) \in \rT \times [0, \bar L]} f(s,i,l) \leq \bar L w + \left(\frac{w}{r}+\chi\right)\gamma \eqqcolon K_f,
\end{equation}
and, for all $(s,i), \, (s',i') \in \rT$,
\begin{equation}\label{eq:flip}
\sup_{l \in [0, \bar L]} \abs{f(s,i,l) - f(s',i',l)} \leq M_f \norm{(s,i) - (s',i')},
\end{equation}
where $M_f \coloneqq 2\left[\bar Lw + \left(\dfrac wr + \chi\right)(\gamma + M_\phi)\right]$ {\color{blue}and $M_\phi$ is the Lipschitz constant appearing in~\eqref{eq:philip}}

From these facts, we deduce the following result.
\begin{proposition}\label{prop:V_UCb}
The value function $V$ given in~\eqref{eq:valuefunction} is non-negative, bounded, and uniformly continuous on $\rT$. If, moreover, $r + \nu \geq M_b$, where $M_b$ is the constant appearing in~\eqref{eq:blip}, then $V$ is Lipschitz continuous on $\rT$.
\end{proposition}

\begin{proof}
The value function $V$ is clearly non-negative, because $f$ is.
Boundedness easily follows from~\eqref{eq:fbounded}. Indeed, for all $L \in \rL$ and all $(s,i) \in \rT$,
\begin{align*}
J(L,s,i)
&= \int_0^{+\infty} \de^{-(r+\nu)t} f(S_t, I_t, L_t) \ud t
\\
&\leq \sup_{(s,i,l) \in \rT \times [0, \bar L]} f(s,i,l) \int_0^{+\infty} \de^{-(r+\nu)t} \, \dd t \leq \dfrac{K_f}{r+\nu},
\end{align*}
and hence
\begin{equation*}
0 \leq V(s,i) = \inf_{L \in \rL} J(L,s,i) \leq \dfrac{K_f}{r+\nu}.
\end{equation*}
To prove uniform continuity it is enough to show that $V$ is continuous on $\rT$, because $\rT$ is a compact subset of $\R^2$. Let us fix $\epsilon > 0$, $(s,i), \, (\tilde s, \tilde i) \in \rT$, and the corresponding solutions to~\eqref{eq:SIRDreduced} $(S,I)$, $(\widetilde S,\widetilde I)$, for any given admissible control. Consider an $\epsilon$-optimal control for the minimization problem with initial data $(\tilde s,\tilde i)$, that is, $L^\epsilon \in \rL$ such that
\begin{equation*}
V(\tilde s,\tilde i) \geq \int_0^{+\infty} \de^{-(r+\nu)t} f(\widetilde S_t, \widetilde I_t, L^\epsilon_t) \, \dd t - \epsilon.
\end{equation*}
Then, for a constant $T > 0$ to be chosen later, using~\eqref{eq:fbounded}, \eqref{eq:flip}, and~\eqref{eq:init_data_SI}, we obtain
\begin{align*}
V(s,i) - V(\tilde s,\tilde i)
&\leq
\int_0^T \de^{-(r+\nu)t} \abs{f(S_t, I_t, L^\epsilon_t) - f(\widetilde S_t, \widetilde I_t, L^\epsilon_t)} \, \dd t
\\
&\qquad \qquad + \int_T^{+\infty} \de^{-(r+\nu)t} \abs{f(S_t, I_t, L^\epsilon_t) - f(\widetilde S_t, \widetilde I_t, L^\epsilon_t)} \, \dd t + \epsilon
\\
&\leq M_f \int_0^T \de^{-(r+\nu)t} \norm{(S_t, I_t) - (\widetilde S_t, \widetilde I_t)} \, \dd t + 2 K_f \int_T^{+\infty} \de^{-(r+\nu)t} + \epsilon
\\
&\leq M_f \int_0^T \de^{-(r+\nu)t} \norm{(s,i)-(\tilde s, \tilde i)} \, \de^{M_b t} \, \dd t + \dfrac{2 K_f}{r+\nu} \de^{-(r+\nu)T} + \epsilon
\\
&\leq \dfrac{M_f}{r+\nu-M_b}(1-\de^{-(r+\nu-M_b)T}) \norm{(s,i)-(\tilde s, \tilde i)} + \dfrac{2 K_f}{r+\nu} \de^{-(r+\nu)T} + \epsilon.
\end{align*}
The second term of the right hand side of the last inequality can be made smaller than $\epsilon$ choosing $T$ large enough, while the first one can be made smaller than $\epsilon$ choosing an appropriate $\delta > 0$ such that $\norm{(s,i)-(\tilde s, \tilde i)} < \delta$.
Exchanging the roles of $(s,i), \, (\tilde s, \tilde i)$, we get that $V$ is continuous on $\rT$.

Finally, if $r + \nu \geq M_b$, we can take $T \to +\infty$ and $\epsilon \to 0^+$ and get from the last inequality
\begin{equation*}
\abs{V(s,i) - V(s',i')} \leq \dfrac{M_f}{r+\nu-M_b} \norm{(s,i)-(\tilde s, \tilde i)},
\end{equation*}
that is, $V$ is Lipschitz continuous on $\rT$.
\end{proof}

\subsection{Dynamic Programming principles}\label{sec:dpprinc}

In this subsection we provide the dynamic programming principles for {\color{blue}optimization problem~\eqref{eq:optpb}.}

We start with the following standard result, stating that the value function $V$ satisfies the \emph{Dynamic Programming Principle}. The proof is analogous to that of~\citep[Proposition~III.2.5]{bardi1997:optctrl} and is thus omitted.

\begin{proposition}\label{prop:dpp}
For all $(s,i) \in \rT$ and all $T > 0$, the value function $V$ verifies
\begin{equation*}
V(s,i) = \inf_{L \in \rL} \left\{\int_0^T \de^{-(r+\nu)t} f(S_t, I_t,L_t) \, \dd t + \de^{-(r+\nu)T} V(S_T, I_T)\right\}.
\end{equation*}
\end{proposition}
The following facts can be easily deduced from the Dynamic Programming Principle (for a slightly different approach, see \citep[Theorems~3.1, 3.2]{flemingrishel1975:soc}).
\begin{proposition}\label{prop:Vproperties}
For all $L \in \rL$, the function
\begin{equation*}
t \mapsto \left\{\int_0^t \de^{-(r+\nu)u} f(S_u, I_u,L_u) \, \dd u + \de^{-(r+\nu)t} V(S_t, I_t)\right\},
\end{equation*}
is non-decreasing and it is constant if and only if $L$ is optimal.
\end{proposition}
As a result of the previous proposition, we obtain the following useful regularity result for the value function $V$ evaluated at optimal trajectories of the system~\eqref{eq:SIRDreduced}.
\begin{corollary}\label{cor:V_diff_opt_traj}
Let $\hat L \in \rL$ be an optimal control and denote by $(\hat S, \hat I)$ the corresponding optimal trajectory of the system~\eqref{eq:SIRDreduced}. Then, for almost every $t \geq 0$, there exists $V^\prime(\hat S_t, \hat I_t)$ and
\begin{equation*}
V^\prime(\hat S_t, \hat I_t) = (r+\nu) V(\hat S_t, \hat I_t) - f(\hat S_t, \hat I_t, \hat L_t), \quad t \geq 0.
\end{equation*}
\end{corollary}

\begin{proof}
We follow closely the arguments given in the proof of Proposition~4.13 in~\citep{frenigozzipignotti2008:optstrat}.
Let us consider the function
\begin{equation}\label{eq:gdef}
g(t) \coloneqq \int_0^t \de^{-(r+\nu)u} f(\hat S_u, \hat I_u,\hat L_u) \, \dd u + \de^{-(r+\nu)t} V(\hat S_t, \hat I_t).
\end{equation}
Since $\hat L$ is an optimal control, we know from Proposition~\ref{prop:Vproperties} that $g$ is constant. Moreover, $g$ is differentiable at all Lebesgue points of $\hat L$, which implies that $g'(t) = 0$, for almost all $t \geq 0$.
From~\eqref{eq:gdef} we deduce that
\begin{equation*}
V(\hat S_t, \hat I_t) = \de^{(r+\nu)t} \left\{g(t) - \int_0^t \de^{-(r+\nu)u} f(\hat S_u, \hat I_u,\hat L_u) \, \dd u\right\}, \quad t \geq 0,
\end{equation*}
and hence, for almost all $t \geq 0$, $V^\prime(\hat S_t, \hat I_t)$ exists and satisfies
\begin{equation*}
V^\prime(\hat S_t, \hat I_t) = (r+\nu)\de^{(r+\nu)t} \left\{g(t) - \int_0^t \de^{-(r+\nu)u} f(\hat S_u, \hat I_u,\hat L_u) \, \dd u\right\} - f(\hat S_t, \hat I_t,\hat L_t), \quad t \geq 0,
\end{equation*}
whence the claim.
\end{proof}

We want to show, next, that a result analogous to Proposition~\ref{prop:dpp} holds for the backward trajectories of the system~\eqref{eq:SIRDreduced}, that are given by the solution $\bfY$ of
\begin{equation}\label{eq:SIRDbwd}
  \mathbf{Y}_t^{L,\bfy_0} = \mathbf{y}_0-\int_t^0 {\bfb}\left(\mathbf{Y}_r^{L,\bfy_0}, L_r\right)\ud r, \quad t < 0,
\end{equation}
for any initial condition $\bfy_0 = (s_0,i_0) \in \rT$ and any {\color{blue}Borel-}measurable function $L \colon (-\infty,0]\to[0,\bar L]$. 
{\color{blue}To denote the solution to~\eqref{eq:SIRDbwd} we will use the notation $\bfY_t^{L,\bfy_0}$ or $\bfY_t^{L,s_0,i_0}$, $t \geq 0$, to stress its dependence on $L$ and $\bfy_0 = (s_0,i_0) \in \rT$.
}
We have to restrict the set of admissible controls for the backward equation~\eqref{eq:SIRDbwd}, as there is no guarantee that the backward trajectories remain in the state space $\rT$. Thus, we define, for any given $(s,i) \in \rT$, the set
\begin{equation*}
  \rL^-(s,i)=\left\{L\colon(-\infty,0]\to[0,\bar L] \text{ Borel-measurable, s.t. } \mathbf{Y}_t^{L,s,i} \in \rT, \, \forall t \in (-\infty,0] \right\}.
\end{equation*}

\begin{remark}\label{rem:Lmenovuoto}
It is important to note that, if $(s,i) \in \rT$ are such that $s+i = 1$, with $i \neq 0$, then $\rL^-(s,i) = \emptyset$. Indeed, from~\eqref{eq:SIRDreduced} we deduce that
\begin{equation*}
S_t + I_t = s + i + \int_t^0 \left(\gamma + \phi(I_t)\right) I_t \, \dd t, \quad t < 0,
\end{equation*}
and hence $S_t + I_t > 1$, for all $t < 0$.
\end{remark}

We need also the following definition.
\begin{definition}\label{def:optimalpoint}
We say that a point $(s,i) \in \rT$ is \emph{optimal} if there exist $(s_0,i_0)\in \rT$, $t^\ast>0$, and an optimal strategy $L^\ast \in \rL$ -- i.e., $V\left(s_0,i_0\right)=J(L^\ast,s_0,i_0)$ -- such that $(s,i)=\mathbf{X}_{t^\ast}^{\left(L^\ast;s_0,i_0\right)}$. We denote the set of optimal points by $\rO$.
\end{definition}

\begin{remark}
The definition above can be rephrased saying that the controller can drive the system starting from $(s_0,i_0)$ to $(s,i)$ in finite time with an optimal strategy $L^\ast$.
\end{remark}


We are now ready to state the \emph{Backward Dynamic Programming Principle}. Its proof is very similar to the one given of~\citep[Proposition~4.7]{frenigozzipignotti2008:optstrat} (see also, \citep[Proposition~2.25]{bardi1997:optctrl}).
\begin{proposition}
  \label{prop:BDP}
  For every $(s,i) \in \rT$, every $t>0$, and every $L\in\rL^-(s,i)$, the value function $V$ satisfies
  \begin{equation*}
    V(s,i)\geq V\left(\mathbf{Y}_{-t}^{\left(L;s,i\right)}\right) \de^{(r+\nu)t}-\int_0^t f\left(\mathbf{Y}_{-u}^{\left(L;s,i\right)},L_{-u}\right) \de^{(r+\nu)u}\ud u\ \ .
  \end{equation*}
  Moreover, if $(s,i)\in\rO$ we have that, for every $0<t<t^\ast$ (where $t^\ast$ is given in Definition~\ref{def:optimalpoint}),
    \begin{equation*}
      V(s,i) = \sup_{L\in\rL^-(s,i)}\left\{V\left(\mathbf{Y}_{-t}^{\left(L;s,i\right)}\right) \de^{(r+\nu)t}-\int_0^t f\left(\mathbf{Y}_{-u}^{\left(L;s,i\right)},L_{-u}\right) \de^{(r+\nu)u}\ud u\right\}\ .
  \end{equation*}
\end{proposition}

\begin{remark}
Since in the second part of Proposition~\ref{prop:BDP} it is assumed that $(s,i)\in\rO$, the supremum in the last equality is attained, and hence we can replace it with a maximum.
\end{remark}

\section{{\color{blue}The Hamilton-Jacobi-Bellman equation}}\label{sec:HJB}

In this section we study the HJB equation for optimization problem~\eqref{eq:optpb}, which is given by


\begin{multline}
  \label{eq:HJB_old}
  (r+\nu)v(s,i)=  \min_{l\in[0,\bar L]}\left[(s+i)lw+\beta(1-\theta l)^2si\left(\partial_iv(s,i)-\partial_sv(s,i)\right)\right]\\+i\phi(i)\left(\frac{w}{r}+\chi\right)-(\gamma+\phi(i))i\partial_iv(s,i), \, \quad (s,i) \in \rT.
 \end{multline}

{\color{blue}
In Section~\ref{sec:Vvisc} we characterize the value function as the unique viscosity solution, in a sense to be made precise later, of the HJB equation~\eqref{eq:HJB_old}. Then, in Section~\ref{sec:opt} we give some optimality conditions, that allow us to characterize optimal policies.

We give a preliminary result that allows to write~\eqref{eq:HJB_old} in a more explicit form.} 
Let us define, for all $(s,i,p,q,l) \in \rT \times \R^2 \times [0,\bar L]$ the \emph{current value Hamiltonian}
\begin{align}\label{eq:HCV}
H_{\mathrm{CV}}(s,i,p,q,l)
&\coloneqq (s+i)lw+\beta(1-\theta l)^2si\left(q-p\right) +i\phi(i)\left(\frac{w}{r}+\chi\right)-(\gamma+\phi(i))iq \notag \\
&= \beta \theta^2 si (q-p) l^2 + \left[(s+i)w - 2\beta \theta si (q-p)\right] l \notag
\\
&\qquad + \beta si (q-p) + i\phi(i)\left(\frac{w}{r}+\chi\right)-(\gamma+\phi(i))iq\, ,
\end{align}
and the \emph{Hamiltonian}
\begin{equation}\label{eq:H}
  H(s,i,p,q) \coloneqq \min_{l\in[0,\bar L]} H_{\mathrm{CV}}(s,i,p,q,l), \quad (s,i,p,q) \in \rT \times \R^2,
\end{equation}
so that~\eqref{eq:HJB_old} can be written as
\begin{equation}\label{eq:HJB}
  (r+\nu)v(s,i)=H\left(s,i,\partial_s v(s,i),\partial_i v(s,i)\right)\ , \quad (s,i) \in \rT.
\end{equation}
We also define the set of minimizers for~\eqref{eq:H}, i.e.,
\begin{equation}\label{eq:Hminimizers}
\Psi(s,i,p,q) \coloneqq \{l \in [0,\bar L] \colon H_{\mathrm{CV}}(s,i,p,q,l)=H(s,i,p,q)\}, \quad (s,i,p,q) \in \rT \times \R^2.
\end{equation}
The following proposition shows that $H$ and $\Psi$ can be explicitly computed.
\begin{proposition}\label{prop:Hexplicit}
Let us define the following sets, which form a partition of $\rT \times \R^2$,
{\color{blue}
\begin{align*}
C_I &\coloneqq \left\{(s,i,p,q) \in \rT \times \R^2 \colon i = 0\right\},
\\
C_S &\coloneqq \left\{(s,i,p,q) \in \rT \times \R^2 \colon s = 0, \, i \neq 0\right\},
\\
A_0 &\coloneqq \left\{(s,i,p,q) \in \rT \times \R^2 \colon s \neq 0, \, i \neq 0, \, p = q\right\},
\\
A_1 &\coloneqq \left\{(s,i,p,q) \in \rT \times \R^2 \colon s \neq 0, \, i \neq 0, \, q < p\right\},
\\
A_2 &\coloneqq \left\{(s,i,p,q) \in \rT \times \R^2 \colon s \neq 0, \, i \neq 0, \, q > p, \, \dfrac{\beta si}{s+i} \leq \dfrac{w}{2\theta(q-p)}\right\},
\\
A_3 &\coloneqq \left\{(s,i,p,q) \in \rT \times \R^2 \colon s \neq 0, \, i \neq 0, \, q > p, \right.
\\
&\qquad \qquad \qquad \qquad \qquad \qquad \left. \dfrac{w}{2\theta(q-p)} < \dfrac{\beta si}{s+i} < \dfrac{w}{2\theta(1-\theta \bar L)(q-p)}\right\},
\\
A_4 &\coloneqq \left\{(s,i,p,q) \in \rT \times \R^2 \colon s \neq 0, \, i \neq 0, \, q > p, \, \dfrac{\beta si}{s+i} \geq \dfrac{w}{2\theta(1-\theta \bar L)(q-p)}\right\}.
\end{align*}
}
For any $(s,i) \in \rT$, $(p,q) \in \R^2$, the Hamiltonian $H$ defined in~\eqref{eq:H} is given by
{\color{blue}
\begin{equation*}
H(s,i,p,q) =
\begin{dcases}
0,	&\text{if } (s,i,p,q) \in C_I, \\
i\phi(i)\left[\dfrac wr + \chi\right] - iq \left(\gamma + \phi(i)\right),	&\text{if } (s,i,p,q) \in C_S \cup A_0, \\
\beta si(q-p) + i\phi(i)\left[\dfrac wr + \chi\right] - iq \left(\gamma + \phi(i)\right), &\text{if } (s,i,p,q) \in A_1 \cup A_2, \\
\frac{w^2}{4\theta^2(p-q)}\frac{(s+i)^2}{\beta si} + \dfrac w\theta(s+i) \\
\qquad  + i\phi(i)\left[\dfrac wr + \chi\right] - iq \left(\gamma + \phi(i)\right), &\text{if } (s,i,p,q) \in A_3, \\
\beta si \left(\theta \bar L -1\right)^2 (q-p) + w \bar L (s+i)  \\
\qquad + i\phi(i)\left[\dfrac wr + \chi\right]- iq \left(\gamma + \phi(i)\right), & \text{if } (s,i,p,q) \in A_4\, .
\end{dcases}
\end{equation*}
}
and the set of minimizers $\Psi$ is given by
{\color{blue}
\begin{equation*}
\Psi(s,i,p,q) =
\begin{dcases}
[0,\bar L], &\text{if } (s,i) = (0,0) \\
\{0\},	&\text{if } (s,i,p,q) \in C_I \cup C_S \cup A_0 \cup A_1 \cup A_2, \\
&\text{ with } (s,i) \neq (0,0), \\
\left\{\frac{1}{\theta}-\frac{w}{2\theta^2(q-p)}\frac{s+i}{\beta si}\right\}, &\text{if } (s,i,p,q) \in A_3, \\
\{\bar L\}, & \text{if } (s,i,p,q) \in A_4.
\end{dcases}
\end{equation*}
}%
Moreover, the Hamiltonian $H$ is continuous on $\rT \times \R^2$ and, for each fixed $(s,i) \in \rT$, the function $\bR^2\ni(p,q)\mapsto H(s,i;p,q)$ is concave.
\end{proposition}

\begin{proof}
We note, first, that continuity of $H$ is a straightforward consequence of the fact that the current value Hamiltonian $H_\mathrm{CV}$ is continuous on $\rT \times \R^2 \times [0,\bar L]$ and that $[0, \bar L]$ is a compact subset of $\R$.

We also note that for each fixed $(s,i,p,q) \in \rT \times \R^2$ the set of minimizers $\Psi(s,i,p,q)$ coincides with the set of minimizers of the quadratic expression
\begin{equation*}
H_0(l; s,i,p,q) \coloneqq \beta \theta^2 si (q-p) l^2 + \left[(s+i)w - 2\beta \theta si (q-p)\right] l, \quad l \in [0,\bar L].
\end{equation*}
We divide our proof according to the different possible cases.

{\color{blue}
Clearly, $H_0(l; 0,0,p,q) = 0$, for all $l \in [0, \bar L]$ and any $(p,q) \in \R^2$, and hence $\Psi(0,0,p,q) = [0,\bar L]$. If, instead, $(s,i) \neq (0,0)$, then
\begin{equation*}
H_0(l; s,i,p,q) =
\begin{dcases}
swl, &\text{if } (s,i,p,q) \in C_I, \\
iwl, &\text{if } (s,i,p,q) \in C_S, \\
(s+i)wl, &\text{if } (s,i,p,q) \in A_0.
\end{dcases}
\end{equation*}
Therefore, for each fixed $(s,i,p,q) \in C_I \cup C_S \cup A_0$, with $(s,i) \neq (0,0)$, the minimum of $l \mapsto H_0(l; s,i,p,q)$, $l \in [0, \bar L]$, is attained at $l = 0$. Thus,
\begin{equation*}
\Psi(s,i,p,q) = \{0\}, \quad (s,i,p,q) \in C_I \cup C_S \cup A_0, \text{ with } (s,i) \neq (0,0),
\end{equation*}
and, still considering $(s,i) \neq (0,0)$,
\begin{equation*}
H(s,i,p,q) = H_{\mathrm{CV}}(s,i,p,q,0) =
\begin{dcases}
0, &\text{if } (s,i,p,q) \in C_I, \\
i\phi(i)\left[\dfrac wr + \chi\right] - iq \left(\gamma + \phi(i)\right), &\text{if } (s,i,p,q) \in C_S \cup A_0.
\end{dcases}
\end{equation*}

Next, we study the case where $s \neq 0$, $i \neq 0$, and $q \neq p$.
}%
We observe that, for each fixed {\color{blue}$(s,i,p,q) \in A_1 \cup A_2 \cup A_3 \cup A_4$}, the abscissa of the vertex of the parabola $x \mapsto H_0(x; s,i,p,q)$ is
{\color{blue}
\begin{equation*}
x^\star \coloneqq \frac{1}{\theta}-\frac{w}{2\theta^2(q-p)}\frac{s+i}{\beta si}\ .
\end{equation*}
}
If $q < p$, then the parabola $x \mapsto H_0(x; s,i,p,q)$ is concave and
{\color{blue} $x^\star > \bar L$. Indeed,
\begin{equation*}
x^\star > \bar L
\quad \Longleftrightarrow \quad
\dfrac{\beta si}{s+i} > \dfrac{w}{2\theta(1-\theta \bar L)(q-p)}.
\end{equation*}
Since $w > 0$, $0 < \theta < 1$, $0 < \bar L \leq 1$, and $q-p < 0$, the right-hand-side is negative, and hence the latter inequality is verified.
Therefore, for each fixed $(s,i,p,q) \in A_1$, the minimum of $l \mapsto H_0(l; s,i,p,q)$, $l \in [0, \bar L]$, is attained at $l = 0$. Thus,
\begin{equation*}
\Psi(s,i,p,q) = \{0\}, \quad (s,i,p,q) \in A_1,
\end{equation*}
and
\begin{equation*}
H(s,i,p,q) = H_{\mathrm{CV}}(s,i,p,q,0) = \beta si(q-p) + i\phi(i)\left[\dfrac wr + \chi\right] - iq \left(\gamma + \phi(i)\right), \, (s,i,p,q) \in A_1.
\end{equation*}
}

If $q > p$, then then the parabola $x \mapsto H_0(x; s,i,p,q)$ is convex and we have three possible cases.
If $x^\star \leq 0$, i.e., if $\frac{\beta si}{s+i} \leq \frac{w}{2\theta(q-p)}$, then the minimum of $l \mapsto H_0(l; s,i,p,q)$, $l \in [0, \bar L]$, is attained at $l = 0$.
{\color{blue} Thus,
\begin{equation*}
\Psi(s,i,p,q) = \{0\}, \quad (s,i,p,q) \in A_2,
\end{equation*}
and
}
\begin{equation*}
H(s,i,p,q) = H_{\mathrm{CV}}(s,i,p,q,0) = \beta si(q-p) + i\phi(i)\left[\dfrac wr + \chi\right] - iq \left(\gamma + \phi(i)\right), \, (s,i,p,q) \in A_2.
\end{equation*}
If, instead, $0 < x^\star < \bar L$, that is, $\frac{w}{2\theta(q-p)} < \frac{\beta si}{s+i} < \frac{w}{2\theta(1-\theta \bar L)(q-p)}$, then the minimum of $l \mapsto H_0(l; s,i,p,q)$, $l \in [0, \bar L]$, is attained at $l = x^\star$.
{\color{blue} Thus,
\begin{equation*}
\Psi(s,i,p,q) = \left\{ \frac{1}{\theta}-\frac{w}{2\beta\theta^2}\frac{s+i}{si}\frac{1}{q-p}\right\}, \quad (s,i,p,q) \in A_3,
\end{equation*}
and
}
\begin{multline*}
H(s,i,p,q) = H_{\mathrm{CV}}(s,i,p,q,x^\star) = \frac{w^2}{4\beta\theta^2}\frac{(s+i)^2}{si}\frac{1}{p-q} - iq \left(\gamma + \phi(i)\right) \\
\qquad + \dfrac w\theta(s+i) +  i\phi(i)\left[\dfrac wr + \chi\right], \quad (s,i,p,q) \in A_3.
\end{multline*}
Finally, if $x^\star \geq \bar L$, that is, $\frac{\beta si}{s+i} \geq \frac{w}{2\theta(1-\theta \bar L)(q-p)}$, then the minimum of $l \mapsto H_0(l; s,i,p,q)$, $l \in [0, \bar L]$, is attained at $l = \bar L$.
{\color{blue} Thus,
\begin{equation*}
\Psi(s,i,p,q) = \{\bar L\}, \quad (s,i,p,q) \in A_4,
\end{equation*}
and
}
\begin{multline*}
H(s,i,p,q) = H_{\mathrm{CV}}(s,i,p,q,\bar L) = \beta\left(\theta \bar L -1\right)^2 si (q-p) - iq \left(\gamma + \phi(i)\right) \\
\qquad + w \bar L (s+i) +  i\phi(i)\left[\dfrac wr + \chi\right], \quad (s,i,p,q) \in A_4.
\end{multline*}
Putting together all these facts we get the explicit expressions for $H$ and $\Psi$.

It is not difficult to show that the expression of $H$ on the set $A_3$ is, for each fixed $(s,i) \in \rT$, a concave function in $(p,q)$. Therefore, $(p,q) \mapsto H(s,i,p,q)$ is concave.
\end{proof}

\begin{remark}
Note that the function $(p,q) \mapsto H(s,i,p,q)$, for fixed $(s,i) \in \rT$, is not strictly concave, as all expressions (except for the third one) given above for $H$ are clearly linear in $(p,q)$.
\end{remark}

\subsection{The value function is the unique viscosity solution of the HJB equation}\label{sec:Vvisc}

We now show that the value function $V$ is the unique solution to the HJB equation~\eqref{eq:HJB} in the sense of viscosity solutions, introduced by \citet{CrandallLions1983:viscsol}. In what follows, if $g \colon \R^n \to \R$ is a continuously differentiable function, $\dD g(x)$ denotes the gradient of $g$ at $x \in \R^n$, i.e., the vector of partial derivatives
\begin{equation*}
\dD g(x) \coloneqq \left(\partial_{x_1} g(x), \dots, \partial_{x_n} g(x)\right), \quad x \in \R^n.
\end{equation*}
We denote by $\mathrm{int}\,\rT$ the interior of $\rT$. We need the following definitions. 
\begin{definition}\label{def:viscsol}
A continuous function $v \colon \rT \to \R$ is called a \emph {constrained viscosity solution} of~\eqref{eq:HJB}
on $\rT$ if it is both
\begin{itemize}
	\item a \emph{viscosity subsolution} of~\eqref{eq:HJB}
 on $\mathrm{int} \, \rT$, i.e., if
	\begin{equation*}
		(r+\nu)v(s,i)-H\left(s,i,\partial_s \psi(s,i),\partial_i \psi(s,i)\right) \leq 0
	\end{equation*}
	whenever $\psi \in \dC^1(\rT)$ and $(s,i) \in \mathrm{int}\, \rT$ is a global maximum point of $(v - \psi)$;
	\item a \emph{viscosity supersolution} of~\eqref{eq:HJB}
 on $\rT$, i.e., if
	\begin{equation*}
		(r+\nu)v(s,i)-H\left(s,i,\partial_s \psi(s,i),\partial_i \psi(s,i)\right) \geq 0
	\end{equation*}
	whenever $\psi \in \dC^1(\rT)$ and $(s,i) \in \rT$ is a global minimum point of $(v - \psi)$.
\end{itemize}
\end{definition}

If $u$ is continuous on an open set $\Omega\subset\bR^n$ and $x\in \Omega$ we define the \emph{superdifferential} of $u$ at $x$ as the set
\begin{equation*}
  \dD^+u(x)=\left\{p\in\bR^n\colon\limsup_{y\to x, y\in \Omega}\frac{u(y)-u(x)-p\cdot(y-x)}{\vert y-x\vert}\leq 0\right\}
\end{equation*}
and the \emph{subdifferential} of $u$ at $x$ as the set
\begin{equation*}
  \dD^-u(x)=\left\{p\in\bR^n\colon\liminf_{y\to x, y\in \Omega}\frac{u(y)-u(x)-p\cdot(y-x)}{\vert y-x\vert}\geq 0\right\}\ .
\end{equation*}

It follows that a continuous function $v\colon\Omega\to\bR$ is a viscosity subsolution of~\eqref{eq:HJB} on $\Omega$ if and only if $	(r+\nu)v(s,i)-H\left(s,i,p,q\right)\leq 0$ for every $(s,i)\in\Omega$ and every $(p,q)\in \dD^+v(s,i)$, and it is a viscosity supersolution of~\eqref{eq:HJB} on $\Omega$ if and only if $(r+\nu)v(s,i)-H\left(s,i,p,q\right)\geq 0$ for every $(s,i)\in\Omega$ and every $(p,q)\in \dD^-v(s,i)$.

We establish, first, the following uniqueness result.
\begin{theorem}\label{th:Vconstrviscsol}
The value function $V$ is the unique constrained viscosity solution to the HJB equation~\eqref{eq:HJB}.
\end{theorem}

\begin{proof}
By Proposition~\ref{prop:V_UCb}, $V$ is a bounded and uniformly continuous function on $\rT$ and, by Proposition~\ref{prop:dpp} it satisfies the Dynamic Programming Principle. Therefore, reasoning as in~\citep[Theorem~2.1]{soner:optcontrol1} (see also~\citep[Theorem~4.10]{calvia:MCcontrol}), we deduce that $V$ is the unique constrained viscosity solution to~\eqref{eq:HJB}.
\end{proof}

Using the Backward Dynamic Programming Principle, we can say something more on the value function as a viscosity solution to~\eqref{eq:HJB}. We need the following definition.

\begin{definition}\label{def:bilateralviscsol}
Let $\Omega \subset \R^n$ be an open set. A continuous function $v \colon \Omega \to \R$ is called a \emph{bilateral viscosity subsolution} (resp. \emph{supersolution}) of~\eqref{eq:HJB} 
on $\Omega$, if and only if $v$ is a viscosity subsolution (resp. supersolution) on $\Omega$ of both
\begin{align}
&(r+\nu)v(s,i)-H\left(s,i,\partial_s v(s,i),\partial_i v(s,i)\right) = 0, \label{eq:bilateral} \\
-&(r+\nu)v(s,i)+H\left(s,i,\partial_s v(s,i),\partial_i v(s,i)\right) = 0, \label{eq:bilateral2}
\end{align}
that is, $(r+\nu)v(s,i)-H\left(s,i,p,q\right)=0$, for every $(s,i)\in\Omega$ and every $(p,q)\in \dD^+v(s,i)$ (resp., $(r+\nu)v(s,i)-H\left(s,i,p,q\right)=0$, for every $(s,i)\in\Omega$ and every $(p,q)\in \dD^-v(s,i)$).

Finally, we say that $v$ is a \emph{bilateral viscosity solution} of~\eqref{eq:HJB}
on $\Omega$ if it is both a bilateral subsolution and supersolution on $\Omega$.
\end{definition}

\begin{remark}
Recall that, in general, {\color{blue}a viscosity solution to either~\eqref{eq:bilateral} or~\eqref{eq:bilateral2} is not a viscosity solution to the other one.}
\end{remark}


\begin{theorem}\label{th:V_viscsol}
The value function $V$ is a bilateral viscosity supersolution to~\eqref{eq:HJB} in the interior of $\rT$. In particular, for every $(p,q)\in \dD^- V(s,i)$, with $(s,i)$ in the interior of $\rT$,
\begin{equation}\label{eq:Vbilatsupersol}
  (r+\nu)V(s,i)=H\left(s,i,p,q\right).
\end{equation}
Moreover, $V$ is a (non bilateral) viscosity subsolution to~\eqref{eq:HJB} in the interior of $\rT$. This is equivalent to say that, for every $(p,q)\in \dD^+ V(s,i)$, with $(s,i)$ in the interior of $\rT$,
\begin{equation*}
  (r+\nu)V(s,i) - H\left(s,i,p,q\right) \leq 0.
\end{equation*}
Finally, for any $(s,i)$ in the boundary of $\rT$ and any $(p,q) \in \dD^-V(s,i)$,
\begin{equation*}
	(r+\nu)V(s,i) - H\left(s,i,p,q\right) \geq 0 \, ,
\end{equation*}
and, for any $(s,i) \in C \cup \rO$, where $\rO$ is the set of optimal points given in Definition~\ref{def:optimalpoint},
\begin{equation*}
C \coloneqq \{(s,i) \in \rT \colon i = 0\} \cup \{(s,i) \in \rT \colon s = 0, \, i \neq 1\},
\end{equation*}
and any $(p,q) \in \dD^+V(s,i)$,
\begin{equation*}
-(r+\nu)V(s,i) + H\left(s,i,p,q\right) \leq 0.
\end{equation*}
\end{theorem}

\begin{proof}
We provide a sketch of the proof.
By Theorem~\ref{th:Vconstrviscsol}, $V$ is the unique constrained viscosity solution to~\ref{th:Vconstrviscsol}, and hence we deduce that $V$ is a viscosity solution of~\eqref{eq:HJB} on the interior of $\rT$. More precisely, we have that, for any $(s,i)$ in the interior of $\rT$,
\begin{align}
(r+\nu)V(s,i) &\geq H\left(s,i,p,q\right), \quad (p,q)\in \dD^- V(s,i), \label{eq:Vviscsupersol} \\
(r+\nu)V(s,i) &\leq H\left(s,i,p,q\right), \quad (p,q)\in \dD^+ V(s,i), \notag
\end{align}
and that, for any $(s,i)$ in the boundary of $\rT$ and any $(p,q) \in \dD^-V(s,i)$,
\begin{equation*}
	(r+\nu)V(s,i) - H\left(s,i,p,q\right) \geq 0 \, .
\end{equation*}
Using the fact that, by Proposition~\ref{prop:BDP}, $V$ satisfies also the Backward Dynamic Programming Principle, arguing as in~\citep[Corollary~III.2.28]{bardi1997:optctrl}, we find that $V$ is a supersolution of $-(r+\nu)V(s,i) + H\left(s,i,\partial_s V(s,i),\partial_i V(s,i)\right) = 0$ on the interior of $\rT$, i.e., for any $(s,i)$ in the interior of $\rT$,
\begin{equation}\label{eq:VsupersolmenoH}
-(r+\nu)V(s,i) + H\left(s,i,p,q\right) \geq 0, \quad (p,q)\in \dD^- V(s,i),
\end{equation}
and that, for any $(s,i) \in \rO$ and any $(p,q) \in \dD^+V(s,i)$,
\begin{equation*}
-(r+\nu)V(s,i) + H\left(s,i,p,q\right) \leq 0.
\end{equation*}
Combining~\eqref{eq:Vviscsupersol} and~\eqref{eq:VsupersolmenoH}, we get~\eqref{eq:Vbilatsupersol}.

We are, thus, left to show that the last inequality in the statement of the theorem holds for all $(s,i) \in C$ and all $(p,q) \in \dD^+V(s,i)$.
Let us define, for all $(s,i)$ in the boundary of $\rT$, the set
\begin{equation*}
\mathcal U(s,i) \coloneqq\{l \in [0, \bar L] \colon \text{ there exist } L \in \rL^-(s,i) \text{ and } \tau > 0 \text{ such that } L_t = l, \, \forall t \in (-\tau, 0]\}.
\end{equation*}
Note that, by Remark~\ref{rem:Lmenovuoto}, we have that $\mathcal U(s,i) = \emptyset$, for all $(s,i) \in \rT$ such that $s+i=1$, with $i \neq 0$. Moreover, a simple inspection of~\eqref{eq:SIRDreduced} reveals that if $i_0 = 0$, then also the backward dynamics is constant, regardless of the choice of $L$, and that if $s_0 = 0$, $i_0 \neq 1$, then for any $l \in [0,\bar L]$ one can find a time $\tau$ and a constant control equal to $l$ on $(-\tau,0]$ that keeps the backward dynamics on the segment $s =0$, $0 < i < 1$. This means that
\begin{equation}\label{eq:UonC}
\mathcal U(s,i) = [0, \bar L], \quad (s,i) \in C.
\end{equation}

Arguing as in~\citep[Proposition~4.10]{frenigozzipignotti2008:optstrat}, we deduce that $V$ verifies, for any $(s,i) \in C$ and any $(p,q) \in \dD^+V(s,i)$,
\begin{equation*}
-(r+\nu)V(s,i) + H_\mathrm{in}\left(s,i,p,q\right) \leq 0,
\end{equation*}
where $H_\mathrm{in}$ is the inward Hamiltonian
\begin{equation*}
H_\mathrm{in}\left(s,i,p,q\right) = \inf_{l \in \mathcal U(s,i)}H_{\mathrm{CV}}(s,i,p,q,l), \quad (s,i,p,q) \in C \times \R^2.
\end{equation*}
By~\eqref{eq:UonC} we deduce that $H_\mathrm{in}\left(s,i,p,q\right) = H\left(s,i,p,q\right)$ for all $(s,i,p,q) \in C \times \R^2$, and hence we obtain the last inequality stated in the theorem.
\end{proof}

\subsection{Optimality conditions}\label{sec:opt}
{\color{blue}
In this section we present some optimality conditions that allow us to interpret optimal strategies in light of the partition of $\rT \times \R^2$ introduced in Proposition~\ref{prop:Hexplicit}.
}
Let us introduce, first, some notations. We define $\dD^{\pm}V(s,i) \coloneqq \dD^+ V(s,i) \cup \dD^- V(s,i)$, $(s,i) \in \rT$ and we denote by $\partial_{(a,b)}V(s,i)$ the directional derivative of $V$ in the direction of the vector $(a,b) \in \R^2$, that is (if the limit exists),
\begin{equation*}
  \partial_{(a,b)} V(s,i)=\lim_{h\to 0}\frac{V(s+ha,i+hb)-V(s,i)}{h\left\Vert(a,b)\right\Vert}\ , \quad (s,i) \in \rT.
\end{equation*}
Clearly, if $(s,i) \in \rT$ is a point in the boundary of $\rT$, we consider only directions $(a,b)$ pointing inside $\rT$.
{\color{blue} To ease notations, whenever $\bfx_0 = (s_0,i_0)\in \rT$ and $L \in \rL$ are fixed, we denote the unique solution to~\eqref{eq:SIRDreduced} $\bfX_t^{L, \bfx_0} = (S_t^{L, s_0,i_0},I_t^{L,s_0,i_0})$, $t \geq 0$, simply by $\bfX_t = (S_t, I_t)$.}

From \citep[Lemma~III.2.50 and Remark~III.2.51]{bardi1997:optctrl} we get that the directional derivatives $\partial_{\mathbf{X}^\prime_t} V\left(\mathbf{X}_t\right)$ exist for almost every $t \geq 0$. Therefore, we get the following conditions for optimality, that can be obtained applying \citep[Theorems~III.2.49, III.2.52]{bardi1997:optctrl} and recalling that, for $r+\nu > M_b$, the value function $V$ is Lipschitz continuous on $\rT$, thanks to Proposition~\ref{prop:V_UCb}.
\begin{theorem}
 \label{thm:optimal}
 Assume that $r+\nu>M_b$ and let $\mathbf{x}_0\in \rT$.
  \begin{enumerate}[label=$(\roman{*})$]
  \item A control strategy $L\in\rL$ is optimal if and only if, for a.e. $t\ge 0$ (i.e. in the Lebesgue points of $f\left(\bfX_t,L_t\right)$),
    \begin{equation*}
      \partial_{\bfX_t^\prime} V \left(\bfX_t\right)+f\left(\bfX_t,L_t\right) = (r+\nu) V\left(\bfX_t\right)\ ,
    \end{equation*}
{\color{blue}where $f$ is the running cost function introduced in~\eqref{eq:runningcost}.}
\item\label{cond:ii} If $L\in\rL$ is an optimal control strategy, then for almost every $t>0$ and every $\mathbf{p}\in \dD^{\pm} V\left(\bfX_t\right)$ we have
 \begin{equation}
   \label{eq:cond_ii}
   \mathbf{p}\cdot \bfb\left(\bfX_t, L_t\right)+f\left(\bfX_t,L_t\right) =\min_{l\in[0,\bar L]}\left\{\mathbf{p}\cdot \bfb\left(\bfX_t, l\right)+f\left(\bfX_t, l\right)\right\}\ .
 \end{equation}
\end{enumerate}
\end{theorem}
\begin{remark}
Condition~\ref{cond:ii} above is also sufficient for optimality if the set $\dD^+V(\mathbf{X}_t)$ coincides with the Clark differential of $V$ at $\mathbf{X}_t$ for almost every $t$. For instance, this is the case if we restrict to constant control strategies, because the value function is then the infimum over a compact set of smooth functions with uniform bounds. The above condition on $\dD^+V$ is also satisfied if the value function happens to be differentiable everywhere.
\end{remark}

Thanks to the explicit computations carried out in Proposition \ref{prop:Hexplicit} we can reformulate condition \ref{cond:ii} in Theorem~\ref{thm:optimal} as follows; we assume $\bfx_0 \neq (0,0)$, $L\in\rL$ optimal, $t>0$ and $\mathbf{p}=(p,q)\in \dD^{\pm} V\left(\bfX_t\right)$.
{\color{blue}
\begin{itemize}
\item If $\left(S_t,I_t,p,q\right)\in C_I \cup C_S\cup A_0 \cup A_1 \cup A_2$, then $L_t=0$.
\item If $\left(S_t,I_t,p,q\right)\in A_3$, then
\begin{equation*}
 L_t=\frac{1}{\theta}-\frac{w}{2\theta^2(q-p)}\frac{S_t+I_t}{\beta S_t I_t}\ ;
\end{equation*}
in particular, $0<L_t<\bar L$.
\item If $\left(S_t,I_t,p,q\right)\in A_4$, then $L_t=\bar L$.
\end{itemize}
Furthermore, if we assume that the value function $V$, given in~\eqref{eq:valuefunction} is differentiable everywhere, then we can interpret optimal strategies and the partition of $\rT \times \R^2$ appearing in Proposition~\ref{prop:Hexplicit} as follows. Assume that $\bfx_0 \neq (0,0)$, that $L \in \sL$ is optimal, consider $t > 0$, and define
\begin{align*}
K^{(1)}(S_t, I_t) &= \frac{w}{2\theta(\partial_i V(S_t, I_t) - \partial_s V(S_t, I_t))}, \\
K^{(2)}(S_t, I_t) &= \frac{w}{2\theta(1-\theta \bar L)(\partial_i V(S_t, I_t) - \partial_s V(S_t, I_t))}.
\end{align*}
\begin{itemize}
\item If $I_t = 0$, then there is no epidemic. In this case, $(S_t, I_t, \partial_s V(S_t, I_t), \partial_i V(S_t, I_t)) \in C_I$ and, clearly, $L_t = 0$.
\item If $S_t = 0$, then the epidemic dies out without any need for a lockdown. In this case, $(S_t, I_t, \partial_s V(S_t, I_t), \partial_i V(S_t, I_t)) \in C_S$ and $L_t = 0$. \item If $S_t, I_t \neq 0$, then the value at time $t$ of the optimal policy $L$ depends also on the derivatives of the value function and, in particular, on the sign of $\partial_i V(S_t, I_t) - \partial_s V(S_t, I_t)$. More precisely, if $\partial_i V(S_t, I_t) - \partial_s V(S_t, I_t) \leq 0$, i.e., if the marginal cost of the infected is not higher than the marginal cost of the susceptibles, then the optimal policy at time $t$ is a \emph{laissez-faire} policy. In this case, $(S_t, I_t, \partial_s V(S_t, I_t), \partial_i V(S_t, I_t)) \in A_0 \cup A_1$. If, instead, $\partial_i V(S_t, I_t) - \partial_s V(S_t, I_t) > 0$, i.e., if the marginal cost of the infected is higher than the marginal cost of the susceptibles, then:
\begin{itemize}
\item The optimal policy at time $t$ is a \emph{laissez-faire} policy whenever the ratio between the rate of newly infected people and the population that can be put in lockdown, $\frac{\beta S_t I_t}{S_t+I_t}$, is not higher than the threshold $K^{(1)}(S_t,I_t)$. In this case, 
\\
\noindent $(S_t, I_t, \partial_s V(S_t, I_t), \partial_i V(S_t, I_t)) \in A_2$;
\item A fraction of the population, smaller than $\bar L$, is put in lockdown at time $t$, whenever the ratio $\frac{\beta S_t I_t}{S_t+I_t}$ is between the two thresholds $K^{(1)}(S_t,I_t)$ and
$K^{(2)}(S_t,I_t)$. In this case, $(S_t, I_t, \partial_s V(S_t, I_t), \partial_i V(S_t, I_t)) \in A_3$;
\item The highest possible fraction of population, i.e., $\bar L$, is put in lockdown at time $t$, whenever the ratio $\frac{\beta S_t I_t}{S_t+I_t}$ is higher than the threshold $K^{(2)}(S_t,I_t)$. In this case, $(S_t, I_t, \partial_s V(S_t, I_t), \partial_i V(S_t, I_t)) \in A_4$.
\end{itemize}
\end{itemize}
}

\begin{remark}\label{rem:estensioni}
  We observe that our main results can be applied to other similar epi-econ model which display the same structure, in particular:
  \begin{itemize}
    \item the state equations are a controlled modification of the compartmental models in epidemiology, like SIR or similar;
    \item the cost functional to minimize is not strictly convex or{\color{blue}, possibly, non-convex.}
  \end{itemize}
  This is the case, for instance, of the model {\color{blue}discussed in}~\citep{acemoglu2020}.
All the results above hold {\color{blue} also in that context, with all the required adaptations}. {\color{blue}A more general setting in which the techniques showed in this paper may be applied, is the optimal} control of age-structured SIR-type models. To the best of our knowledge, the {\color{blue}study} of HJB equations in this context is still not carried out completely (see, e.g.,~\citep{FabbriGozziZanco2021}).

{\color{blue} As already happens in the case discussed in this paper, also in these similar or more general models} some open issues remain. {\color{blue}For instance, a deeper study of optimal strategies is required and this calls for} different ideas and proofs.
\end{remark}

\section{Conclusion}\label{sec:concl}
This paper makes a first step towards a complete analysis
of the dynamic programming approach for a class of epi-econ models
that have been formulated and studied in recent years.
From a technical point of view such models are difficult to study mainly due to the lack of convexity of the dynamics and of the cost.
{\color{blue}Existing numerical methods for solutions to HJB equations in the viscosity sense are not suitable (nor can be straightforwardly adapted) to simulate the value function of our optimization problem. Such simulations, in the absence of a closed-form expression for the value function, would allow us to obtain more insights about its behaviour.%
}
{\color{blue}Other important aspects that we could not analyze with the results presented here are the existence and uniqueness of an optimal strategy, possibly in feedback form, and the behavior of} optimal trajectories{\color{blue}, that is, the evolution of the epidemics under the action of an optimal control.}

Nevertheless, we think that {\color{blue}our results} provide a solid ground for further
research. {\color{blue}For instance, an important aspect to be analyzed is the behavior of} optimal trajectories.
{\color{blue}More precisely}, the next steps will be:
\begin{itemize}
\item to characterise the set where the value function is differentiable and where singularities in its gradient may arise;
\item to use the sufficient optimality conditions proved here to characterise the optimal strategies;
\item {\color{blue}to extend or adapt existing numerical schemes to the non-convex case, in order to cover at least some of the examples mentioned herein.}
\end{itemize}

\section*{Declarations}

\noindent\textbf{Competing interests.} The authors have no competing interests to declare that are relevant to the content of this article.

\medskip

\noindent\textbf{Funding.} F.~Gozzi and F.~Lippi acknowledge financial support from the ERC grant 101054421-DCS.
 Views and opinions expressed are however those of the author(s) only and do not necessarily reflect those of the European Union or the European Research Council.  

A.~Calvia, F.~Gozzi, G.~Zanco are supported by the Italian Ministry of University and Research (MIUR), in the framework of PRIN project 2017FKHBA8 001 (The Time-Space Evolution of Economic Activities: Mathematical Models and Empirical Applications).

\bibliographystyle{plainnat}
\bibliography{Bibliography}

\end{document}